\documentclass{raex}
\usepackage{amsmath,amssymb,amsthm}

\newtheorem{theorem}{Theorem}

\newtheorem{example}[theorem]{Example}
\newtheorem{corollary}[theorem]{Corollary}
\newtheorem{lemma}[theorem]{Lemma}
\theoremstyle{definition}
\newtheorem{definition}[theorem]{Definition}
\newtheorem{remark}[theorem]{Remark}

\begin{Author} 
	\FirstName{Hajrudin}\LastName{Fejzic}
	\PostalAddress{Department of Mathematics, California State University, San Bernardino, CA 92407, U.S.A}
	\Email{hfejzic@csusb.edu }
	 
\end{Author}
\begin{MathReviews} 
	\primary{26A24} 
	\secondary{26A27}
\end{MathReviews}

\begin{KeyWords} 
\end{KeyWords}

\title{The Marcinkiewicz-Zygmund Property for Riemann Differences with Geometric Nodes}

\begin{document}
\maketitle
\begin{abstract}
We study when a Riemann difference of order \( n \) possesses the 

\noindent Marcinkiewicz-Zygmund (MZ) property: that is, whether the conditions \( f(h) = o(h^{n-1}) \) and \( Df(h) = o(h^n) \) imply \( f(h) = o(h^n) \). This implication is known to hold for some classical examples with geometric nodes, such as \( \{0, 1, q, \dots, q^{n-1}\} \) and \( \{1, q, \dots, q^n\} \), leading to a conjecture that these are the only such Riemann differences with the MZ property (see \cite{ACF1}). However, this conjecture was disproved by the third-order example with nodes \( \{-1, 0, 1, 2\} \) (\cite{CF2}), and we provide further counterexamples and a general classification here.

We establish a complete analytic criterion for the MZ property by developing a recurrence framework: we analyze when a function \( R(h) \) satisfying \( D(h) = R(qh) - A R(h) \), together with \( D(h) = o(h^n) \) and \( R(h) = o(h^{n-1}) \), forces \( R(h) = o(h^n) \). We prove that this holds if and only if \( A \) lies outside a critical modulus annulus determined by \( q \) and \( n \), covering both \( |q| > 1 \) and \( |q| < 1 \) cases. This leads to a complete characterization of all Riemann differences with geometric nodes that possess the MZ property, and provides a flexible analytic framework applicable to broader classes of generalized differences.
\end{abstract}

\maketitle

\section{Introduction}

The concept of differentiation lies at the heart of analysis. To address limitations in smoothness and extend differentiation to broader function classes, various generalizations of the classical derivative have been developed. Among these, Riemann derivatives offer a discrete framework for computing higher-order derivatives directly from function values. These derivatives depend on a choice of interpolation nodes and associated coefficients, and their properties reveal subtle relationships between analytic and measure-theoretic structures.

Let $n \geq 1$ be fixed. For distinct real numbers $b_1, b_2, \dots, b_{n+1}$ (called \emph{nodes}), the associated divided difference operator is the linear operator $\Delta_n$ defined on the space of real-valued functions $f$ by
\[
\Delta_n f = \sum_{k=1}^{n+1} \frac{f(b_k)}{\prod_{j \neq k} (b_k - b_j)} = \sum_{k=1}^{n+1} a_k f(b_k),
\]
where the weights $a_k = \frac{1}{\prod_{j \neq k} (b_k - b_j)}$ depend only on the nodes. These weights coincide with the Lagrange interpolation coefficients corresponding to the nodes $b_k$.

The corresponding difference operator is defined by
\[
R_h f(c) := \sum_{k=1}^{n+1} a_k f(c + b_k h),
\]
so that the $n$-th \emph{Riemann derivative} of $f$ at a point $c$ is given by
\[
Rf^{(n)}(c) = \lim_{h \to 0} \frac{n! R_h f(c)}{h^n},
\]
provided the limit exists. Notably, different choices of nodes give rise to different Riemann derivatives. For instance, the \emph{forward Riemann derivative} corresponds to nodes $0,1,2,\dots,n$ and is given by
\[
Rf^{(n)}(c) = \lim_{h \to 0} \frac{\sum_{k=0}^{n} (-1)^{n-k} \binom{n}{k} f(c + kh)}{h^n},
\]
while a symmetric choice of nodes produces the \emph{symmetric Riemann derivative}.

The weights $a_k$ associated to a Riemann derivative satisfy the system
\begin{equation} \label{q1}
\sum_{k=1}^{n+1} a_k b_k^j = 
\begin{cases}
0 & \text{for } j = 0, 1, \dots, n-1, \\
1 & \text{for } j = n.
\end{cases}
\end{equation}
Thus, whenever the classical $n$-th derivative $f^{(n)}(c)$ exists, so does $Rf^{(n)}(c)$, and the two coincide. This positions Riemann derivatives as natural tools for discrete differentiation. However, the converse implication
\[
Rf^{(n)}(c) \Rightarrow f^{(n)}(c)
\]
does not hold in general. A simple counterexample is the Schwartz derivative:
\[
f_s(c) = \lim_{h \to 0} \frac{f(c+h) - f(c-h)}{2h},
\]
which exists at $c = 0$ for $f(x) = |x|$ but $f$ fails to be differentiable at $0$.

Peano observed that the existence of a polynomial $p(x)$ satisfying
\[
f(c+h) - p(c+h) = o(h^n)
\]
does not imply the existence of $f^{(n)}(c)$. Functions satisfying this approximation property are said to be \emph{$n$-times Peano differentiable} at $c$, and the coefficients of $\frac{(x-c)^k}{k!}$ in $p(x)$ are denoted $f_k(c)$. We always assume $f_0(c) = f(c)$, so that Peano differentiability of order zero is continuity and of order one is ordinary differentiability. The relation
\[
f_n(c) \Rightarrow Rf^{(n)}(c)
\]
follows from Taylor’s theorem and the structure of \eqref{q1}.

For $n \geq 2$, the converse implication
\[
Rf^{(n)}(c) \Rightarrow f^{(n)}(c)
\]
can fail everywhere for non-measurable functions, and even on large sets for measurable functions. We present an example illustrating the first case and show that the second arises from the failure of the converse to a generalized form of Taylor’s theorem.

For each nowhere dense closed set $H \subset \mathbb{R}$ and $n \geq 2$, there exists a function $f$ such that $f_n$ exists everywhere but $f^{(n)}(c)$ fails precisely on $H$. Thus, even for measurable functions, the implication $Rf^{(n)}(c) \Rightarrow f^{(n)}(c)$ fails dramatically. This naturally raises the question: does $Rf^{(n)}(c)$ imply $f_n(c)$?

A theorem of Marcinkiewicz and Zygmund affirms this implication under the assumption of measurability: if $f$ is measurable on $\mathbb{R}$ and $Rf^{(n)}(c)$ exists on a measurable set $E$, then $f_n(c)$ exists almost everywhere on $E$. Their argument is based on a more local result: if $f_0(c),\dots,f_{n-1}(c)$ exist, and the Riemann derivative corresponding to the nodes $\{0,1,2,4,\dots,2^{n-1}\}$ exists at $c$, then $f_n(c)$ exists.

This leads to the main focus of the present paper: we investigate which Riemann differences possess the property that
\[
f_k(c) \text{ exists for } 0 \leq k \leq n-1 \text{ and } Rf^{(n)}(c) \text{ exists } \Rightarrow f_n(c) \text{ exists}.
\]
We call such differences those with the \emph{Marcinkiewicz-Zygmund (MZ) property}. If this implication fails for a given $R$, then we can construct a function $g$ such that $g_k(0) = 0$ for $k < n$, $R^{(n)}g(0) = 0$, but $g_n(0)$ does not exist. Thus, the MZ property is equivalent to the implication
\[
g(h) = o(h^{n-1}) \text{ and } Rg(h) = o(h^n) \Rightarrow g(h) = o(h^n),
\]
where $Rg(h) := \sum_k a_k g(b_k h)$ is a generalized difference operator.

Previous results have shown that Riemann differences with geometric nodes of the form $\{0,1,q,\dots,q^{n-1}\}$ or $\{1,q,\dots,q^n\}$ for $q \notin \{-1,0,1\}$ satisfy the MZ property. This led to the conjecture that only such geometric-node differences could have the MZ property. However, this conjecture was disproven by the third-order example with nodes $\{-1,0,1,2\}$, which also satisfies the property (see \cite{CF2}).

In this paper, we establish a complete classification of all Riemann differences with geometric nodes that possess the MZ property. We introduce a general analytic framework based on the recurrence
\[
D(h) = R(qh) - A R(h),
\]
and determine precisely when such a relation, combined with decay conditions $D(h) = o(h^n)$ and $R(h) = o(h^{n-1})$, forces $R(h) = o(h^n)$. Our main result shows that this happens if and only if $|A|$ lies outside a critical annulus determined by $q$ and $n$. The classification applies uniformly across the cases $|q| > 1$ and $0 < |q| < 1$, and leads to a necessary and sufficient condition for the MZ property in terms of the roots of the characteristic polynomial of the difference.

We also demonstrate how this framework resolves various previously open examples, including several fourth-order Riemann differences that do satisfy the MZ property despite not conforming to earlier geometric templates. The techniques extend to asymmetric and shifted configurations and show that the MZ property is governed by the root geometry rather than the symmetry or placement of nodes.

Finally, we present an application illustrating how continuity at the origin can enlarge the class of first-order Riemann differences equivalent to the ordinary derivative. In particular, we show that continuity can restore differentiability in cases where, without it, the classical implication fails (cf. \cite{ACC}).
\section{Necessary conditions}

In the introduction, we observed that the existence of a Riemann derivative does not imply the existence of the ordinary derivative, and that this failure can persist even when the function is measurable. We then asked a more refined question: under what conditions does the existence of a Riemann derivative imply the existence of the corresponding Peano derivative? This leads to the Marcinkiewicz-Zygmund property, which asserts that such an implication holds under appropriate regularity assumptions.

In this section, we present examples showing that the Peano implication can fail when those assumptions are violated. These examples demonstrate that continuity, measurability, and the existence of lower-order Peano derivatives are necessary for the MZ property to hold. They also motivate the formal framework we will use to classify which differences possess this property.

Marcinkiewicz and Zygmund proved that if $f$ is measurable on $\mathbb{R}$ and the Riemann derivative $Rf^{(n)}(x)$ exists on a measurable set $E$, then the $n$-th Peano derivative $f_n(x)$ exists almost everywhere on $E$. In particular, if $f$ is Schwartz differentiable on $\mathbb{R}$—i.e., if the symmetric difference quotient converges—then $f$ is also ordinarily differentiable almost everywhere. Remarkably, this conclusion holds even without assuming measurability of $f$. However, as the following example demonstrates, the measurability hypothesis becomes essential when $n \geq 2$.

Let $n \geq 2$ and fix a set of nodes $\{b_1, b_2, \ldots, b_{n+1}\}$. We will construct a non-linear solution of the Cauchy functional equation
\begin{equation} \label{ch1}
f(x+y) = f(x) + f(y)
\end{equation}
such that $Rf^{(n)}(x) = 0$ for all $x \in \mathbb{R}$. By a result of Darboux, such a function is nowhere continuous and, in particular, is not Peano differentiable at any point. That such a solution is also non-measurable was first established independently by Banach and Sierpiński in 1920.

\begin{example} \label{e1}
Let $F = \mathbb{Q}(b_1, b_2, \ldots, b_{n+1})$ be the smallest subfield of $\mathbb{R}$ containing $\mathbb{Q}$ and the nodes. Since $F$ is countable, we can choose real numbers $\alpha$ and $\beta$ such that $\alpha/\beta \notin F$. Then $\{\alpha, \beta\}$ is linearly independent over $F$, and we may extend it to a Hamel basis $H$ of $\mathbb{R}$ over $F$.

Every $x \in \mathbb{R}$ has a unique finite representation
\[
x = d(x)\alpha + \sum_i d_i h_i,
\]
where $d(x), d_i \in F$ and $h_i \in H$. Define $f(x) = d(x)$, the coefficient of $\alpha$ in the Hamel expansion of $x$. Then $f$ satisfies \eqref{ch1}, since for any $x,y \in \mathbb{R}$,
\[
f(x+y) = d(x+y) = d(x) + d(y) = f(x) + f(y).
\]
Because $f(\alpha) = 1$ and $f(\beta) = 0$, $f$ is not linear (indeed $f(0) = 0$), and hence $f$ is nowhere continuous and non-measurable.

For any $c \in \mathbb{R}$ and $h \in \mathbb{R}$, we have $d(c + b_k h) = d(c) + b_k d(h)$, so that
\[
\sum a_k f(c + b_k h) = d(c) \sum a_k + d(h) \sum a_k b_k = 0
\]
by the system \eqref{q1}. Hence $Rf^{(n)}(c) = 0$ for all $c \in \mathbb{R}$.
\end{example}

This function shows that the implication $Rf^{(n)}(c) \Rightarrow f_n(c)$ can fail everywhere when $f$ is not continuous at $c$. Thus, for the implication to hold, it is necessary that $f$ be continuous at $c$. More generally, as the next example illustrates, the existence of the first $n-1$ Peano derivatives at $c$ is also necessary.

\begin{example}
Let $n \geq 2$, let $\{b_1, \ldots, b_{n+1}\}$ be a set of nodes, and let $F$ be the field from Example~\ref{e1}. Define
\[
f(c + h) =
\begin{cases}
h^{n-1} & \text{if } h \in F, \\
0 & \text{otherwise}.
\end{cases}
\]
Then $f$ is measurable and satisfies $f(c+h) = o(h^{n-2})$ as $h \to 0$. By the conditions \eqref{q1}, we again have $Rf^{(n)}(c) = 0$. However, $f_{n-1}(c) = (n-1)!$ on $F$ and $0$ on its complement, so $f_n(c)$ fails to exist.
\end{example}

 Suppose there exists a function $f$ and a point $c$ such that $f_k(c)$ exists for $k = 0, \dots, n-1$, and $Rf^{(n)}(c)$ exists. Define a function
\[
g(h) = f(c + h) - \sum_{k=0}^{n-1} \frac{f_k(c)}{k!} h^k - \frac{Rf^{(n)}f(c)}{n!} h^n.
\]
Then $g_k(0) = 0$ for all $k < n$, and $R^{(n)}g(0) = 0$ by the system \eqref{q1}. This is equivalent to the asymptotic conditions \[ g(h)=o(h^{n-1}) \text{ and } Rg(h)=o(h^n)\] where $Rg(h)=\sum_k a_k g(b_k h)$. Moreover, $f_n(c)$ exists if and only if $g_n(0)$ exists  in which case $g(h)=o(h^{n})$. This observation motivates the following general definition of the Marcinkiewicz-Zygmund property.

\begin{definition} Let $\{a_k\}_{k=1}^m$ be a finite  sequence of complex numbers,  and let $\{b_k\}_{k=1}^m
$ be a finite sequence of distinct real numbers. The associated linear difference operator $R$ is defined by
\[
Rg(h) = \sum_k a_k g(b_k h)
\] where the $a_k$ are called coefficients, and the $b_k$ are the nodes of the operator.

\noindent Let $n\geq 1$. We say that $R$ has the Marcinkiewicz–Zygmund (MZ) property of order $n$ if, for all real-valued functions $g$ the conditions \[g(h) = o(h^{n-1}) \text{ and } Rg(h) = o(h^n)\] imply \[g(h) = o(h^n).\] Here,  \( f(h) = o(h^n) \) means $\lim_{h \to 0} \frac{|f(h)|}{|h|^n} = 0.$

\end{definition}

\noindent\textbf{Notation and Clarification.} 
In what follows, we fix a positive integer $n$, so references to the MZ property will omit the explicit mention of order $n$, which is understood from context.

 We distinguish between two related uses of the symbol \( R \). When we write \( Rf(h) \), we mean the application of a linear difference operator \( R \), defined by
\[
Rf(h) := \sum_k a_k f(b_k h),
\]
where \( \{a_k, b_k\} \) are fixed coefficients and nodes. If these coefficients satisfy the moment system~\eqref{q1}, then \( R \) is called a \emph{Riemann difference} of order \( n \).

More generally, we say that a difference operator \( R \) is of \emph{order \( n \)} if it annihilates all polynomials of degree less than \( n \), but acts nontrivially on \( x^n \):
\[
Rf(h) = 0 \quad \text{for all polynomials } f(x) = x^m \text{ with } 0 \leq m \leq n-1,\]\[ \quad \text{and} \quad R(x^n)(h) = c h^n \text{ with } c \neq 0.
\]

In contrast, when we write \( R(h) \), we refer to an abstract complex-valued function of \( h \), which may or may not arise from applying a difference operator to a function. Theorems involving recursions like \( D(h) = R(qh) - A R(h) \) operate at this more general level, where both \( R(h) \) and \( D(h) \) are understood as functions satisfying asymptotic and recurrence properties. This distinction allows us to formulate results in maximum generality while still applying them to concrete difference operators.

\medskip
\noindent\textbf{Normalization.}
When classifying differences up to the MZ property, it is convenient to reduce to a standard form. The following lemma justifies this reduction.

\begin{lemma}
Let \( R \) be a difference with coefficients \( a_k \) and nodes \( b_k \), and let \( c \neq 0 \), \( r \neq 0 \). Define the rescaled difference \( R' \) by
\[
R'f(h) := \sum_k r a_k f\left( cb_k h \right)=rRf(ch).
\]
Then \( R \) has the MZ property if and only if \( R' \) does.
\end{lemma}

\begin{proof}
The change of variables \( h \mapsto ch \) implies \( f(h) = o(h^m) \) if and only if \( f(ch) = o(h^m) \). Rescaling the coefficients by \( r \) does not affect asymptotic decay, so the MZ property is preserved.
\end{proof}

As a consequence, we may assume without loss of generality that one of the nodes is equal to \( 1 \), and that one of the coefficients is equal to \( 1 \).

\medskip

\section{Recurrence criteria for asymptotic decay}

In this section, we develop a general analytic framework to determine when a function \( R(h) \colon \mathbb{R} \to \mathbb{C} \), satisfying a recurrence of the form
\[
D(h) = R(qh) - A R(h),
\]
exhibits higher-order vanishing at zero. This setting arises naturally in the study of the Marcinkiewicz-Zygmund property for Riemann differences, where \( R(h) \) corresponds to a difference operator applied to a function and \( D(h) \) represents its increment.

Our goal is to determine when the conditions
\[
|D(h)| = o(h^n) \quad \text{and} \quad |R(h)| = o(h^{n-1})
\]
force the improved decay \( |R(h)| = o(h^n) \). We show that this occurs if and only if the recurrence coefficient \( A \) lies outside a critical annulus bounded by \( |q|^{n-1} \) and \( |q|^n \). The results apply to all real numbers \( q \neq 0, \pm 1 \).

\medskip
\noindent\textbf{Overview }
\begin{itemize}
    \item In the low-growth and high-growth regimes, we prove that \( R(h) = o(h^n) \) necessarily holds.
    \item In the intermediate range \( |q|^{n-1} < |A| \leq |q|^n \), ($|q|^n\leq |A|<|q|^n$ if $|q|<1$) we construct explicit counterexamples where \( R(h) = o(h^{n-1}) \) and \( D(h) = o(h^n) \) but \( R(h) \not= o(h^n) \).
    \item The low-growth and failure cases are relatively straightforward, but the high-growth case requires delicate control of recurrence growth and asymptotic tracking.
\end{itemize}

We will provide the proofs only for the cases \( |q| > 1 \), since the case \( |q| < 1 \) follows from this one by the substitution
\[
R'(h) := R(qh), \quad D'(h) := \frac{-1}{A} D(h),
\]
under which the recurrence \( D(h) = R(qh) - A R(h) \) is transformed into the equivalent relation
\[
D'(h) = R'\left(\frac{1}{q} h\right) - \frac{1}{A} R'(h).
\]

\subsection{Decay when \( A \) lies outside the critical annulus}

\begin{theorem}[Low-growth case]\label{th1}
Let \( |q| > 1 \), and let \( D(h) \), \( R(h) \) be complex-valued functions such that
\[
|D(h)| = o(h^n), \quad |R(h)| = o(h^{n-1}),
\]
and
\[
D(h) = R(qh) - A R(h),
\]
for some \( A \in \mathbb{C} \) with \( |A| \leq |q|^{n-1} \). Then \( R(h) = o(h^n) \).
\end{theorem}
\begin{proof}
Fix \( \epsilon > 0 \). Since \( D(h) = o(h^n) \), there exists \( \delta > 0 \) such that \( |D(h)| < \epsilon |h|^n \) whenever \( 0 < |h| < \delta \). Fix such an \( h \in (0, \delta) \), and define the sequence \( h_k := h / q^k \) for \( k \in \mathbb{N} \).

By the recurrence relation,
\[
R(h_k) = \frac{R(h_{k-1}) - D(h_k)}{A},
\]
we can expand \( R(h) \) as a telescoping sum:
\[
R(h) = A^k R(h_k) + \sum_{j = 0}^{k - 1} A^j D(h_{j+1}).
\]
Taking absolute values and using \( |D(h)| < \epsilon |h|^n \), we obtain
\[
|R(h)| \leq |A|^k |R(h_k)| + \epsilon \sum_{j = 0}^{k - 1} |A|^j |h_{j+1}|^n.
\]

Note that \( |h_k| = |h| / |q|^k \), so \( |h_{j+1}|^n = |h|^n / |q|^{n(j+1)} \), and
\[
\sum_{j = 0}^{k - 1} |A|^j |h_{j+1}|^n =\frac{|h|^n}{|q|^n} \sum_{j = 0}^{k - 1} \left( \frac{|A|}{|q|^n} \right)^j.
\]
Since \( |A| \leq |q|^{n-1} < |q|^n \), the geometric series converges, and we get
\[
|R(h)| \leq |A|^k |R(h_k)| + C \epsilon |h|^n
\]
for some constant \( C \) independent of \( h \).

We now estimate the term \( |A|^k |R(h_k)| \). Write
\[
|A|^k |R(h_k)| = |h_k|^{n-1} \left| \frac{R(h_k)}{h_k^{n-1}} \right| |A|^k = |h|^{n-1} \left( \frac{|A|}{|q|^{n-1}} \right)^k \left| \frac{R(h_k)}{h_k^{n-1}} \right|.
\]
Since \( |A| / |q|^{n-1} \leq 1 \), and \( \frac{R(h_k)}{h_k^{n-1}} \to 0 \), we conclude that
\[
|A|^k |R(h_k)| \to 0 \quad \text{as } k \to \infty.
\]
Therefore, from the bound
\[
|R(h)| \leq |A|^k |R(h_k)| + C \epsilon |h|^n,
\]
we conclude \( R(h) = o(h^n) \) as desired.

\end{proof}

\noindent The proof of the low-growth case is a standard telescoping argument. In contrast, the high-growth case requires more intricate asymptotic control:

\begin{theorem}[High-growth case]\label{th2}
Let \( |q| > 1 \), and let \( D(h) \), \( R(h) \) be complex-valued functions such that
\[
|D(h)| = o(h^n), \quad |R(h)| = o(h^{n-1}),
\]
and
\[
D(h) = R(qh) - A R(h),
\]
for some \( A \in \mathbb{C} \) with \( |A| > |q|^n \). Then \( R(h) = o(h^n) \).
\end{theorem}

\begin{proof}
Suppose not. Then there exists a constant \( d > 0 \) and a sequence \( h_k \to 0 \) such that
\[
\frac{|R(h_k)|}{|h_k|^n} > d.
\]

For each \( h_k \), define \( m_k \in \mathbb{N} \) to be the unique integer such that
\[
\frac{|q|^{(n-1)(m_k+1)}}{|A|^{m_k+1}} < |h_k| \leq \frac{|q|^{(n-1)m_k}}{|A|^{m_k}}.
\]
The existence of such \( m_k \) is guaranteed because the function \( m \mapsto \frac{|q|^{(n-1)m}}{|A|^m} \) is strictly decreasing and tends to 0 as \( m \to \infty \), given that \( |A| > |q|^{n-1} \). Importantly, it is the **first inequality** above that  implies \( m_k \to \infty \) as \( h_k \to 0 \).

Fix \( \epsilon > 0 \). Since \( D(h) = o(h^n) \), there exists \( \delta > 0 \) such that
\[
|D(h)| < \epsilon |h|^n \quad \text{whenever } 0 < |h| < \delta.
\]

Choose \( M \in \mathbb{N} \) so that
\[
\left( \frac{|q|^n}{|A|} \right)^M < \delta.
\]
Then for all sufficiently large \( k \), we have \( m_k \geq M \), and all points \( h_k^{(j)} := q^j h_k \) for \( j = 0, \dots, m_k \) satisfy
\[
|h_k^{(j)}| \leq |q|^{m_k} |h_k| \leq \left( \frac{|q|^n}{|A|} \right)^M < \delta.
\]
Thus, we may apply the recurrence at all those points.

Now iterate the recurrence:
\[
R(h_k^{(j+1)}) = A R(h_k^{(j)}) + D(h_k^{(j)}),
\]
which yields
\[
R(h_k^{(m_k)}) = A^{m_k} R(h_k) + \sum_{j=0}^{m_k - 1} A^{m_k - 1 - j} D(h_k^{(j)}).
\]

Taking absolute values:
\[
\left| R(h_k^{(m_k)}) - A^{m_k} R(h_k) \right| \leq \sum_{j=0}^{m_k - 1} |A|^{m_k - 1 - j} |D(h_k^{(j)})|.
\]

Since \( |D(h_k^{(j)})| < \epsilon |q|^{jn} |h_k|^n \), we get:
\[
\left| R(h_k^{(m_k)}) - A^{m_k} R(h_k) \right|
< \epsilon |h_k|^n \sum_{j=0}^{m_k - 1} |A|^{m_k - 1 - j} |q|^{jn}
\]

\[
= \epsilon |h_k|^n |A|^{m_k - 1} \sum_{j=0}^{m_k - 1} \left( \frac{|q|^n}{|A|} \right)^j
\leq C \epsilon |A|^{m_k - 1} |h_k|^n,
\]
where \( C = \sum_{j=0}^{\infty} \left( \frac{|q|^n}{|A|} \right)^j < \infty \), since \( |A| > |q|^n \).

Now divide both sides by \( |h_k^{(m_k)}|^{n-1} = |q|^{(n-1)m_k} |h_k|^{n-1} \):
\[
\left| \frac{R(h_k^{(m_k)})}{|h_k^{(m_k)}|^{n-1}} \right|
\geq \left| \frac{A^{m_k} R(h_k)}{|q|^{(n-1)m_k} |h_k|^{n-1}} \right|
- C \epsilon |h_k| \left( \frac{|A|}{|q|^{n-1}} \right)^{m_k - 1}.
\]

Now use the **lower bound** from the first inequality in the definition of \( m_k \):
\[
|h_k| > \frac{|q|^{(n-1)(m_k+1)}}{|A|^{m_k+1}}.
\]
Therefore:
\[
|h_k| \left( \frac{|A|}{|q|^{n-1}} \right)^{m_k - 1}
> \frac{|q|^{(n-1)(m_k+1)}}{|A|^{m_k+1}} \cdot \left( \frac{|A|}{|q|^{n-1}} \right)^{m_k - 1}
= \frac{|q|^{2(n-1)}}{|A|^2}.
\]

Also, since \( \frac{|R(h_k)|}{|h_k|^n} > d \), we have:
\[
\left| \frac{A^{m_k} R(h_k)}{|q|^{(n-1)m_k} |h_k|^{n-1}} \right|
> d |h_k| \left( \frac{|A|}{|q|^{n-1}} \right)^{m_k}.
\]

Putting everything together:
\[
\left| \frac{R(h_k^{(m_k)})}{|h_k^{(m_k)}|^{n-1}} \right|
\geq \left( d \cdot \frac{|A|}{|q|^{n-1}} - C \epsilon \right)
\cdot |h_k| \left( \frac{|A|}{|q|^{n-1}} \right)^{m_k - 1}.
\]

And by the above,
\[
|h_k| \left( \frac{|A|}{|q|^{n-1}} \right)^{m_k - 1}
> \frac{|q|^{2(n-1)}}{|A|^2}.
\]
So finally:
\[
\left| \frac{R(h_k^{(m_k)})}{|h_k^{(m_k)}|^{n-1}} \right|
> \left( d \cdot \frac{|A|}{|q|^{n-1}} - C \epsilon \right)
\cdot \frac{|q|^{2(n-1)}}{|A|^2}.
\]

Choose \( \epsilon \) small enough so that the factor in parentheses is positive. Then the right-hand side is a positive constant independent of \( k \), while the left-hand side must tend to 0 by the assumption \( R(h) = o(h^{n-1}) \). This contradiction completes the proof.
\end{proof}

\subsection{Failure when \( A \) lies inside the critical annulus}

\begin{theorem}\label{th3}
Assume \( |q| > 1 \), and let \( A \in \mathbb{C} \) satisfy \( |q|^n \geq |A| > |q|^{n-1} \). Then there exists a complex-valued function \( R\) such that
\[
|R(h)| = o(h^{n-1}), \quad |R(h)| \not= o(h^n),
\]
and
\[
R(qh) = A R( h).
\]
\end{theorem}
 \begin{proof}
 Let $a=\frac{A}{q^{n-1}}$ and define $$R(h)=\left\{ \begin{array}{cc} a^kh^{n-1} & \text{ if } h=q^k \text{ for some integer $k$} \\  0 & \text {otherwise.}\end{array}\right.$$
 
 If $h=q^k$ and  $h\to 0$ then $k\to -\infty$. The condition $|q|^{n}\geq |A|>|q|^{n-1}$ implies that $|a|>1$ and $\frac{|a|}{|q|}\leq 1$. 
 
 It follows that  $$\frac{|R(h)|}{|h|^{n-1}}=|a|^k\to 0, \text{ and }\frac{|R(h)|}{|h|^{n}}=\left(\frac{|a|}{|q|}\right)^k\not \to 0$$ as $h\to 0$.   Hence $|R(h)|=o(h^{n-1})$, and $|R(h)|\neq o(h^n)$. 
 
It remains to show that $R(qh)=AR(h)$, for $h=q^k$. To that end, $$R(qh)-AR(h)=a^{k+1}q^{n-1}h^{n-1}-Aa^{k}h^{n-1}= a^{k}h^{n-1}(aq^{n-1}-A)=0.$$ 
 \end{proof}

\subsection{Unified trichotomy}

\begin{corollary}\label{cor:trichotomy}
Let \( q \in \mathbb{R} \setminus \{0, \pm1\} \), and let \( D(h), R(h) \) be complex-valued functions satisfying
\[
D(h) = R(qh) - A R(h),
\]
with \( D(h) = o(h^n) \) and \( R(h) = o(h^{n-1}) \). Define the closed annulus
\[
\mathcal{A}_q := 
\begin{cases}
\{ z \in \mathbb{C} : |q|^{n-1} < |z| \leq |q|^n \}, & \text{if } |q| > 1, \\\\
\{ z \in \mathbb{C} : |q|^n \leq |z| < |q|^{n-1} \}, & \text{if } 0 < |q| < 1.
\end{cases}
\]
Then:
\begin{itemize}
  \item If \( |A| \notin \mathcal{A}_q \), then \( R(h) = o(h^n) \).
  \item If \( |A| \in \mathcal{A}_q \), then there exists \( R(h) \) such that \( D(h) = o(h^n) \), \( R(h) = o(h^{n-1}) \), but \( R(h) \not= o(h^n) \).
\end{itemize}
\end{corollary}

\subsection{An elementary application}

We conclude this section with an illustrative example showing how the recurrence condition characterizes differentiability.

\begin{theorem}
Let \( f \colon \mathbb{R} \to \mathbb{R} \) be continuous at \( 0 \), with \( f(0) = 0 \), and fix \( A \in \mathbb{R} \). Suppose
\[
\lim_{h \to 0} \frac{f(2h) - A f(h)}{h} = 0.
\]
Then \( f'(0) = 0 \) exists if and only if \( A \notin (1,2] \).
\end{theorem}

\begin{proof}
Set \( D(h) := f(2h) - A f(h) \), and \( R(h) := f(h) \). Then \( D(h) = R(2h) - A R(h) \), and the hypothesis gives \( D(h) = o(h) \), with \( R(h) = o(1) \).

Applying the recurrence trichotomy with \( n = 1 \), \( q = 2 \), we conclude:
\begin{itemize}
  \item If \( A \notin (1, 2] \), then \( R(h) = o(h) \), so \( f'(0) = 0 \).
  \item If \( A \in (1, 2] \), construct the counterexample:
  \[
  f(h) :=
  \begin{cases}
  a^k, & \text{if } h = 2^k \, (k \in \mathbb{Z}), \\
  0, & \text{otherwise},
  \end{cases} \quad \text{where } a = A.
  \]
  Then \( f(2h) = A f(h) \), so \( D(h) = 0 = o(h) \), but \( f(h) \not= o(h) \), and \( f'(0) \) does not exist.
\end{itemize}
\end{proof}

\subsection*{Summary}

This section provides a complete classification of when the recurrence \( D(h) = R(qh) - A R(h) \) forces higher-order decay under minimal assumptions. The resulting trichotomy, parameterized by the location of \( |A| \) relative to \( |q|^{n-1} \) and \( |q|^n \), will be crucial in our later classification of Riemann differences with the Marcinkiewicz-Zygmund property. The section also includes an elementary application to first-order difference quotients, illustrating how continuity at the origin can expand the class of operators that imply differentiability.
Notably, the final application refines the known characterization of first-order Riemann differences equivalent to the ordinary derivative (valid for arbitrary functions; see \cite{ACC}) by showing that when continuity at the origin is assumed, the class of differences that imply differentiability expands. In particular, the example \( f(h) = a^k \) on dyadic points demonstrates that continuity can restore the implication of differentiability even for differences that fail the full-function criterion.

\section{Classification for geometric-node differences}

Let \( q \in \mathbb{R} \setminus \{0, \pm 1\} \), and let \( \{a_k\}_{k \geq 0} \) be a finite sequence of complex coefficients, not all zero. Define \( a := -\sum a_k \), and consider the difference operator \( D \) induced by the node-coefficient pairs \( \{(a_k, q^k)\} \cup \{(a, 0)\} \), acting on functions \( f \colon \mathbb{R} \to \mathbb{C} \) by
\[
Df(h) := a f(0) + \sum_{k \geq 0} a_k f(q^k h).
\]
We associate to \( D \) the characteristic polynomial
\[
p(x) := \sum_{k \geq 0} a_k x^k,
\]
so that \( a = -p(1) \) and \( D(1) = 0 \) automatically holds. Our goal is to classify such differences according to whether they possess the Marcinkiewicz-Zygmund property, i.e., whether
\[
f(h) = o(h^{n-1}) \quad \text{and} \quad Df(h) = o(h^n) \quad \Rightarrow \quad f(h) = o(h^n),
\]
for all (real- or complex-valued) functions \( f \).

We now give a complete classification of when the MZ property holds for such differences, in terms of the location of the roots of the characteristic polynomial. This general result includes all Riemann differences supported on geometric progressions \( \{ q^k \} \cup \{0\} \), and applies equally to real or complex coefficients.

Our main tool is the following decomposition lemma.

\begin{lemma}\label{ll1}
Let \( A \) be a root of the characteristic polynomial \( p(x) = \sum a_k x^k \), and write \( p(x) = (x - A) r(x) \). Let \( R \) be the difference induced by the node-coefficient pairs \( \{(r_k, q^k)\} \), where \( r_k \) are the coefficients of \( r(x) \). Then:
\[
Df(h) = Rf(qh) - A Rf(h).
\]
\end{lemma}

\begin{proof}
Using the identity \( p(x) = (x - A) r(x) = \sum_{k \geq 1} (r_{k-1} - A r_k) x^k - A r_0 \), we obtain
\[
Df(h) = \sum_{k \geq 1} (r_{k-1} - A r_k) f(q^k h) - A r_0 f(h),
\]
which rearranges as
\[
Df(h) = \sum_{k \geq 0} r_k f(q^{k+1} h) - A \sum_{k \geq 0} r_k f(q^k h) = Rf(qh) - A Rf(h).
\]
\end{proof}

We now state the main result of this section.

\begin{theorem}\label{m}
Let \( q \in \mathbb{R} \setminus \{0, \pm 1\} \), and let \( D \) be the difference operator associated to the characteristic polynomial \( p(x) = \sum a_k x^k \), with node-coefficient pairs \( \{(a_k, q^k)\} \cup \{(-p(1), 0)\} \). Fix an integer \( n \geq 1 \). Then \( D \) has the Marcinkiewicz-Zygmund property of order \( n \) if and only if \( p(x) \) has no roots in the closed annulus
\[
\left\{ z \in \mathbb{C} : |q|^{n-1} < |z| \leq |q|^n \right\}.
\]
\end{theorem}

\begin{proof}
Suppose first that \( p(x) \) has a root \( A \) with \( |q|^{n-1} < |A| \leq |q|^n \). By Theorem~\ref{th3}, there exists a complex-valued function \( f(h) \) such that \( f(h) = o(h^{n-1}) \), \( f(h) \not= o(h^n) \), and \( f(qh) = A f(h) \). Using this in \( Df(h) = \sum a_k f(q^k h) = \sum a_k A^k f(h) = p(A) f(h) = 0 \), we conclude that \( D \) fails the MZ property.

Now assume that all roots of \( p(x) \) lie either inside \( |z| \leq |q|^{n-1} \) or outside \( |z| > |q|^n \). We proceed by induction on the degree of \( p(x) \).

If \( p(x) = a_0(x - A) \), then \( Df(h) = a_0(f(qh) - A f(h)) \), and by Theorem~\ref{th1} or~\ref{th2}, we conclude \( f(h) = o(h^n) \).

In the general case, write \( p(x) = (x - A) r(x) \), and define the corresponding difference \( R \) as in Lemma~\ref{ll1}. Since \( f(h) = o(h^{n-1}) \), linearity gives \( Rf(h) = o(h^{n-1}) \). From Lemma~\ref{ll1} we have \( Df(h) = Rf(qh) - A Rf(h) = o(h^n) \).

By Theorem~\ref{th1} or~\ref{th2}, it follows that \( Rf(h) = o(h^n) \), and by the induction hypothesis (applied to \( r(x) \)), we conclude \( f(h) = o(h^n) \).
\end{proof}

\begin{remark}
Although the proof above allows complex-valued functions, the counterexample construction in the failing case can always be reduced to a real-valued function by taking real or imaginary parts. Hence, Theorem~\ref{m} also classifies when the MZ property holds for real-valued functions.
\end{remark}

\begin{corollary}\label{cor:symmetric-geo}
Let \( q \in \mathbb{R} \setminus \{0, \pm 1\} \), and let \( D(h) \) be a difference of the form
\[
Df(h) = a f(0) + \sum_{k \geq 0} \left( a_k f(q^k h) + a_{-k} f(-q^k h) \right),
\]
with \( a_k, a_{-k} \in \mathbb{C} \), and where \( a = -\sum_{k \geq 0} (a_k + a_{-k}) \). Let
\[
r^+(x) := \sum_{k \geq 0} (a_k + a_{-k}) x^k, \quad r^-(x) := \sum_{k \geq 0} (a_k - a_{-k}) x^k.
\]
Then \( D \) has the Marcinkiewicz-Zygmund property of order \( n \) if and only if all roots of \( r^+(x) \) and \( r^-(x) \) lie outside the closed annulus
\[
\left\{ z \in \mathbb{C} : |q|^{n-1} < |z| \leq |q|^n \right\}.
\]
\end{corollary}

\begin{proof}
Apply the identity:
\[
Df(h) = \frac{1}{2} \left[ D(h) + D(-h) \right] + \frac{1}{2} \left[ D(h) - D(-h) \right],
\]
and observe:
\[
D(h) + D(-h) = 2a f(0) + \sum_{k \geq 0} (a_k + a_{-k}) \left[ f(q^k h) + f(-q^k h) \right],
\]
\[
D(h) - D(-h) = \sum_{k \geq 0} (a_k - a_{-k}) \left[ f(q^k h) - f(-q^k h) \right].
\]
Define \( f_{\mathrm{even}}(h) := f(h) + f(-h) \) and \( f_{\mathrm{odd}}(h) := f(h) - f(-h) \). Then:
\[
\frac{1}{2}(D(h) + D(-h)) = D^+f_{\mathrm{even}}(h), \quad \frac{1}{2}(D(h) - D(-h)) = D^- f_{\mathrm{odd}}(h),
\]
where \( D^+ \), \( D^- \) are differences with characteristic polynomials \( r^+(x) \), \( r^-(x) \) respectively.

Now apply Theorem~\ref{m} to each of \( D^+ \) and \( D^- \). Since \( f(h) = o(h^{n-1}) \) implies that both \( f_{\mathrm{even}}(h) \) and \( f_{\mathrm{odd}}(h) \) are \( o(h^{n-1}) \), and since
\[
Df(h) = o(h^n) \quad \Rightarrow \quad D^+ f_{\mathrm{even}}(h) = o(h^n), \quad D^- f_{\mathrm{odd}}(h) = o(h^n),
\]
we conclude \( f(h) = o(h^n) \)  if both \( r^+(x) \) and \( r^-(x) \) have no roots in the annulus \( |q|^{n-1} < |z| \leq |q|^n \).

Now suppose that $r^+(x)=\sum c_jx^j$ has a root $A$ inside \( |q|^{n-1} < |z| \leq |q|^n \). Let  $a=\frac{A}{q^{n-1}}$, and define $f(h)=a^kh^{n-1}$ if $h=q^k$ for some $k\in \mathbb{Z}$, $f(-h)=f(h)$ and 0 otherwise. Then $f(h)=o(h^{n-1})$ and $f(h) \neq o(h^n)$. We also have $f_{even}(h)=2a^kh^{n-1}$ if $h=q^k$ for some $k\in \mathbb{Z}$ and 0 otherwise, $f_{odd}(h)=0$ and 

\[Df(h)=D^+f_{even}(h)=2\sum c_j \frac{A^{k+j}}{q^{(k+j)(n-1)}}q^{j(n-1)}h^{n-1}=\]
\[2\sum c_j A^j\frac{A^k}{q^{k(n-1)}}h^{n-1}=2r^+(A)\frac{A^k}{q^{k(n-1)}}h^{n-1}=0.\]

 Simillarly if $r^-(x)=\sum c_jx^j$ has a root $A$ inside \( |q|^{n-1} < |z| \leq |q|^n \), we can construct $f(h)=o(h^{n-1})$ and $f(h) \neq o(h^n)$ and such that $Df(h)=0$.
\end{proof}

\subsection{Applications to explicit Riemann differences}

We now illustrate Theorem~\ref{m} and Corollary~\ref{cor:symmetric-geo} with explicit examples of rescaled Riemann differences of order \( n = 4 \), each supported on five geometric nodes. These examples highlight the central case in which the number of nodes is exactly one more than the order --- the minimal configuration required for a difference of order \( n \). The MZ property in such cases can be verified directly from the root location conditions established earlier.

\begin{example}
Let \( q = 2 \), and consider the difference
\[
Df(h) = f(32h) - \frac{620}{7} f(8h) + 1984 f(2h) - \frac{23040}{7} f(h) + 1395 f(0).
\]
This is a rescaled Riemann difference of order \( 4 \) with nodes \( \{32, 8, 2, 1, 0\} \). The characteristic polynomial is
\[
p(x) = x^5 - \frac{620}{7} x^3 + 1984 x - \frac{23040}{7},
\]
whose roots are
\[
\{2, 4, 8, \tfrac{-49 \pm \sqrt{119}i}{7} \}.
\]
The complex roots have magnitude approximately \( \frac{6}{7} \sqrt{70} \approx 7.17 \), which is less than \( 2^3 = 8 \). Since all roots lie outside the critical annulus \( 2^3 < |z| \leq 2^4 \), Theorem~\ref{m} implies that \( D \) satisfies the MZ property.
\end{example}

\begin{example}
Let \( q = 2 \), and consider the difference
\[
Df(h) = f(4h) - 10 f(2h) + 20 f(h) - 15 f(0) + 4 f(-h),
\]
a rescaled Riemann difference of order \( 4 \) with nodes \( \{-1, 0, 1, 2, 4\} \). Define the symbol polynomials:
\[
r^+(x) = x^2 - 10x + 24 = (x - 4)(x - 6), \quad r^-(x) = x^2 - 10x + 16 = (x - 2)(x - 8).
\]
All roots of \( r^+(x) \) and \( r^-(x) \) lie outside the annulus \( 2^3 < |z| \leq 2^4 \). Hence, by Corollary~\ref{cor:symmetric-geo}, this difference also satisfies the MZ property.
\end{example}

\begin{remark}
These examples demonstrate how Theorem~\ref{m} and Corollary~\ref{cor:symmetric-geo} can be applied to verify the MZ property in concrete cases. Differences of order \( n \) with \( n+1 \) geometric nodes form a particularly natural and fundamental class, and their classification under the MZ property is a central focus of this paper.
\end{remark}

\section{Riemann differences with MZ property}

It is known that Riemann differences of order $n$ with the nodes $\{0, 1, q, \ldots, q^{n-1}\}$ and those with the nodes $\{1, q, \ldots, q^n\}$, for $q \notin \{-1, 0, 1\}$, satisfy the MZ property. It was once conjectured that these were the only Riemann differences with the MZ property. (see \cite{ACF1}) This conjecture was supported by the failure of symmetric Riemann differences to satisfy the MZ property and by recent results showing that forward Riemann differences for $n \geq 3$ also fail the MZ property. (see \cite{CF1})

However, the conjecture is false: a third-order Riemann difference with nodes $\{-1, 0, 1, q\}$ satisfies the MZ property. In the previous section, we provided two examples of fourth-order Riemann differences with five nodes that satisfy the MZ property and are not of the two classical forms mentioned above.

If one drops the requirement that a difference be Riemann, it is easy to construct differences of any order that satisfy the MZ property, as we now illustrate.

\begin{lemma}\label{l13}
A difference $Df(h) = a f(0) + \sum a_k f(q^k h)$, with $a = -\sum a_k$, is of order $n$ if and only if its characteristic polynomial is of the form
\[
p(x) = \prod_{j = 1}^{n - 1}(x - q^j) \cdot t(x),
\]
where $t(x)$ is a polynomial such that $t(q^n) \neq 0$.
\end{lemma}

\begin{proof}
By definition, $D$ is of order $n$ if and only if $D(x^j)(h) \equiv 0$ for $j = 0, 1, \ldots, n - 1$ and $D(x^n)(h) \not\equiv 0$. The condition for $j = 0$ is equivalent to $a = -p(1)$.

For $j \geq 1$, we compute
\[
D(x^j)(h) = \left( \sum a_k (q^k)^j \right) h^j = p(q^j) h^j.
\]
So $p(q^j) = 0$ for $j = 1, \ldots, n - 1$, and $p(q^n) \neq 0$, which implies the stated factorization.
\end{proof}

Lemma~\ref{l13} can be used to generate differences of order $n$. For example, to construct a difference of order $3$ with nodes generated by $q = 2$, we take any polynomial $q(x)$ such that $q(8) \neq 0$ and define
\[
p(x) = (x - 2)(x - 4)q(x) = \sum a_k x^k,
\]
then set $Df(h) = -p(1) f(0) + \sum a_k f(2^k h)$. To ensure that $D$ satisfies the MZ property, choose $q(x)$ so that $p(x)$ has no roots in the interval $(4, 8]$.

An obvious question is whether all Riemann differences with geometric nodes satisfy the MZ property. The following example shows this is not the case.

\begin{example}
Consider the difference
\[
Df(h) = f(16h) - 120 f(2h) + 224 f(h) - 105 f(0),
\]
with nodes $\{0, 1, 2, 2^4\}$. The associated characteristic polynomial is
\[
p(x) = x^4 - 120x + 224,
\]
whose roots are $\{2, 4, -3 \pm i \sqrt{19}\}$, with the complex roots having modulus approximately $5.29$. Since $4 < |z| < 8$ for these roots, the difference fails the MZ property.
\end{example}

Next, we considerably enlarge the class of Riemann differences that satisfy the MZ property.

Each such difference will have nodes from the set \(\{-q^s,-q^{s-1},\ldots,-1,0,1,\)

\noindent \(q,q^2,\ldots, q^k\}\) and it will take on the form
\[
Df(h) = a f(0) + \sum_{k = 0}^t \left( a_k f(q^k h) + a_{-k} f(-q^k h) \right),
\]
where we assume $a_{-k} = 0$ for $s < k \leq t<n$, and for simplicity, we take $a_t = 1$ to ensure $r^+(x)$ and $r^-(x)$ are monic. By Corollary~\ref{cor:symmetric-geo}, the MZ property reduces to a condition on the roots of the even and odd symbol polynomials:
\[
r^+(x) = \sum_{k \geq 0} (a_k + a_{-k}) x^k, \quad r^-(x) = \sum_{k \geq 0} (a_k - a_{-k}) x^k.
\]

We begin by identifying the structure of the roots of $r^+$ and $r^-$:

\begin{lemma}\label{l5.4}
Let $D$ be a difference of order $n$ with nodes in $\{\pm q^k : k \in \mathbb{N}_0\} \cup \{0\}$. Then:
\begin{itemize}
  \item $q^{2j}$ is a root of $r^+(x)$ for all even $2j < n$,
  \item $q^{2j+1}$ is a root of $r^-(x)$ for all odd $2j + 1 < n$.
\end{itemize}
\end{lemma}

\begin{proof}
By definition, $D(x^j)(h) = 0$ for $j < n$. Using symmetry and linearity:
\[
D(x^{2j})(h) + D(x^{2j})(-h) = 2 h^{2j} r^+(q^{2j}) = 0,
\]
\[
D(x^{2j+1})(h) - D(x^{2j+1})(-h) = 2 h^{2j+1} r^-(q^{2j+1}) = 0.
\]
\end{proof}

\begin{lemma}[Alternating interpolation zeros]\label{l1}
Let \( c_1 < c_2 < \cdots < c_{m+1} \), and let \( p \) be the unique polynomial of degree at most \( m \) satisfying
\[
p(c_i) =
\begin{cases}
0, & \text{if } i \text{ is odd}, \\
1, & \text{if } i \text{ is even}.
\end{cases}
\]
Then \( \deg(p) = m \), and there exist points \( d_1 < d_2 < \cdots < d_{m-1} \) in \( (c_1, c_{m+1}) \) such that
\[
p(d_i) =
\begin{cases}
1, & \text{if } i \text{ is odd}, \\
0, & \text{if } i \text{ is even}.
\end{cases}
\]
\end{lemma}

\begin{proof}
If \( m = 1 \), then \( p \) is linear, and the conclusion is vacuously true. Assume \( m \geq 2 \). Since \( p \) interpolates \( m+1 \) values at distinct points, it must have degree at least \( m \).

We construct the intermediate values \( d_i \in (c_1, c_{m+1}) \) where \( p \) alternates between 0 and 1. These arise from the structure of \( p \) and standard properties of interpolating polynomials.

First, observe that the values where \( p = 1 \) occur at even-indexed points \( c_{2}, c_{4}, \dots \), and the values where \( p = 0 \) occur at odd-indexed points \( c_1, c_3, \dots \). We use the intermediate value theorem and local extrema of \( p \) to find the additional alternating values.

For each even index \( c_{2k} \), where \( 1 \leq k \leq \lfloor (m+1)/2 \rfloor - 1 \), we consider the interval \( (c_{2k-1}, c_{2k+1}) \). If \( p'(c_{2k}) = 0 \), we set \( d_{2k-1} := c_{2k} \); otherwise, the extremum condition ensures that \( p \) reaches a local maximum strictly greater than 1 in the interior, so there exists a point \( d_{2k-1} \in (c_{2k-1}, c_{2k+1}) \) such that \( p(d_{2k-1}) = 1 \).

Next, between each pair \( d_{2k-1} \) and \( d_{2k+1} \), we find a point \( d_{2k} \) such that \( p(d_{2k}) = 0 \).  If \( p'(c_{2k+1}) = 0 \), we may take \( d_{2k} = c_{2k+1} \); otherwise, the minimum of \( p \) on that interval is below 0, and the intermediate value theorem gives a zero.

The last \( d_i \) is constructed similarly using \( d_{m-2} \) and \( c_{m+1} \) to close the sequence. The resulting points \( d_1 < d_2 < \cdots < d_{m-1} \) alternate in value as required and lie strictly between \( c_1 \) and \( c_{m+1} \).

Finally the degree of $p$ is $m$ since we found $m$ zeros counting multiplicity.
\end{proof}

\begin{lemma}[Moment symmetry]\label{l2}
Let \( c_1 < c_2 < \cdots < c_{m+1} \), and let \( s \in \{m-1, m-3, \dots\} \). Then the system of equations
\[
\sum_{\text{odd } i} x_i^j + \sum_{\text{even } i} c_i^j = \sum_{\text{even } i} x_i^j + \sum_{\text{odd } i} c_i^j, \quad \text{for } j = 1, \dots, s,
\]
has a real solution \( x_1, \ldots, x_s \) with each \( x_i \in (c_1, c_{m+1}) \). Moreover, any two solutions differ only by separate permutations of the variables with odd and even indices.
\end{lemma}

\begin{proof}
We argue by reverse induction on \( s \).

If \( s = m - 1 \), we define \( x_i = d_i \), where \( d_1 < \cdots < d_{m-1} \) are the interior interpolation points guaranteed by Lemma~\ref{l1} for the polynomial \( p \) satisfying \( p(c_i) = 0 \) for odd \( i \), and \( p(c_i) = 1 \) for even \( i \). Then the zeros of \( p \) and \( p - 1 \) lie at alternating subsets of \( \{c_i\} \cup \{x_i\} \), and their Newton sums (i.e., power sums of roots) match up to degree \( m - 1 \), yielding a solution to the moment symmetry system of length \( s = m - 1 \).

Now assume that the system has a solution for some \( s \in \{m - 1, m - 3, \dots\} \), with strictly increasing points \( x_1 < \cdots < x_s \in (c_1, c_{m+1}) \). Apply Lemma~\ref{l1} to the \( x_i \) to construct \( s - 2 \) strictly increasing interior points \( y_1 < \cdots < y_{s-2} \in (x_1, x_s) \), satisfying
\[
\sum_{\text{odd } i} y_i^j + \sum_{\text{even } i} x_i^j = \sum_{\text{even } i} y_i^j + \sum_{\text{odd } i} x_i^j, \quad \text{for } j = 1, \dots, s - 2.
\]
Now add this system to the original moment symmetry system (with variables \( x_i \) and \( c_i \)). The \( x_i \) terms cancel due to symmetry of parity, and we obtain
\[
\sum_{\text{odd } i} y_i^j + \sum_{\text{even } i} c_i^j = \sum_{\text{even } i} y_i^j + \sum_{\text{odd } i} c_i^j, \quad \text{for } j = 1, \dots, s - 2.
\]
Thus, the \( y_i \) solve the system of length \( s - 2 \), completing the inductive step.

To prove uniqueness up to separate permutations of odd- and even-indexed variables, suppose \( \{x_i\} \) and \( \{y_i\} \) are two solutions of length \( s \). Let \( p(x) \) be the monic polynomial whose roots are the odd-indexed \( x_i \)'s and the even-indexed \( y_i \)'s, and let \( q(x) \) be the monic polynomial whose roots are the even-indexed \( x_i \)'s and the odd-indexed \( y_i \)'s. Summing the two systems shows that the Newton sums of \( p \) and \( q \) agree up to degree \( s \), so by Newton's identities, \( p = q \), and the roots coincide as multisets. Hence, \( \{x_i\} \) and \( \{y_i\} \) are related by separate permutations of the odd- and even-indexed components.
\end{proof}

\begin{lemma}\label{l99}
Let \( m \) and \( t \) be positive integers. Let \( r^+(x) \) and \( r^-(x) \) be two monic polynomials of degree \( t \) that differ by a polynomial of degree \( l \). Suppose that there are points \( c_1 < c_2 < \cdots < c_{m+1} \), such that \( c_i \) is a root of \( r^+(x) \) when \( i \) is even, and a root of \( r^-(x) \) when \( i \) is odd. If \( t + l = m \), then all of the remaining roots of \( r^+(x) \) and \( r^-(x) \) are real, distinct, and lie in \( (c_1, c_{m+1}) \).
\end{lemma}

\begin{proof}
Let \( s = 2t - m - 1 \), and let \( \{x_i\}_{i=1}^{s} \) denote the remaining roots of \( r^+(x) \) and \( r^-(x) \), labeled so that odd-indexed \( x_i \) are roots of \( r^+(x) \) and even-indexed \( x_i \) are roots of \( r^-(x) \). The condition \( t + l = m \) implies \( s = t - l - 1 \).

Since \( r^+(x) \) and \( r^-(x) \) are both monic and differ by a polynomial of degree \( l \), their top \( t - l \) coefficients agree. By Newton's identities, this implies that the first \( s \) power sums of the roots of \( r^+ \) and \( r^- \) also agree, yielding the moment symmetry system
\[
\sum_{\text{odd } i} x_i^j + \sum_{\text{even } i} c_i^j = \sum_{\text{even } i} x_i^j + \sum_{\text{odd } i} c_i^j, \quad \text{for } j = 1, \dots, s.
\]
By Lemma~\ref{l2}, the \( x_i \) must be real, distinct, and lie in \( (c_1, c_{m+1}) \).
\end{proof}
\begin{theorem}\label{t99}
Let $D$ be a Riemann difference of order $n$ with $n + 1$ nodes from the set
\[
\{-q^l, -q^{l-1}, \ldots, -1, 0, 1, q, \ldots, q^t\},
\]
where $0 \leq l < t$. If $0$ is among the nodes, assume $l + t + 3 = n + 1$; if not, assume $l + t + 2 = n + 1$. Then $D$ satisfies the MZ property.
\end{theorem}

\begin{proof}
By Lemma~\ref{l5.4}, the known roots $q^k$ with $k < n$ are roots of $r^+$ or $r^-$ depending on parity. If $0$ is not among the nodes, then $1$ is a root of $r^+$. In this case, $l + t + 2 = n + 1$, and we set $m + 1 = n$.

If $0$ is among the nodes, then $l + t + 3 = n + 1$, and we set $m + 1 = n - 1$. In either case, setting $c_1 = 1$ (if $0$ is not a node) or $c_1 = q$ (if it is), the number of remaining roots is controlled by Lemma~\ref{l99}. That lemma shows all remaining roots lie in $(c_1, q^{n-1})$, hence outside the critical annulus.

By Corollary~\ref{cor:symmetric-geo}, the MZ property holds.
\end{proof}

\subsection{Illustrative examples: success and failure under shifting}

We conclude this section with two examples that illustrate the scope and limitations of Theorem~\ref{t99} and Corollary~\ref{cor:symmetric-geo}. These cases show that while the MZ property can hold in asymmetric or shifted configurations, it is not preserved by node translation alone.

\begin{example}[Asymmetrically shifted difference satisfying MZ]
Consider the fourth-order Riemann difference with nodes \( \{-2, 1, 2, 4, 8\} \), based on the geometric progression \( q = 2 \). Solving the moment system yields the coefficients
\[
a_{-1} = \tfrac{1}{720}, \quad a_0 = -\tfrac{1}{63}, \quad a_1 = \tfrac{1}{48}, \quad a_2 = -\tfrac{1}{144}, \quad a_3 = \tfrac{1}{1680}.
\]
The associated symbol polynomials are:
\[
r^+(x) = \tfrac{1}{1680}x^3 - \tfrac{1}{144}x^2 + \tfrac{1}{45}x - \tfrac{1}{63},
\]
\[
r^-(x) = \tfrac{1}{1680}x^3 - \tfrac{1}{144}x^2 + \tfrac{7}{360}x - \tfrac{1}{63}.
\]
The roots of \( r^+(x) \) are \( 1, 4, 6.\overline{6} \), and those of \( r^-(x) \) are \( 1.\overline{6}, 2, 8 \). Since none of these lie in the critical annulus \( (2^3, 2^4] = (8, 16] \), Corollary~\ref{cor:symmetric-geo} confirms that \( D \) satisfies the MZ property.
\end{example}

\begin{example}[Translation does not preserve the MZ property]
Consider the third-order Riemann difference with nodes \( \{0, 1, 2, 3\} \). This is a forward difference shifted one unit to the right. It satisfies the order conditions, but fails the MZ property, as shown in \cite{ACF}. In contrast, the centered difference with nodes \( \{-1, 0, 1, 2\} \) does satisfy the MZ property. Thus, shifting the support while preserving the order can destroy the MZ property, emphasizing the importance of root location over node configuration.
\end{example}
\begin{remark}
The results in this section provide a complete classification of when a Riemann difference with geometric nodes satisfies the Marcinkiewicz-Zygmund property. By reducing the problem to the location of the roots of the characteristic polynomial (or its even and odd symmetrizations), we obtain a sharp and verifiable criterion. The classification underscores the central role of root location over node symmetry or translation, and confirm that the MZ property is sensitive to subtle algebraic features of the difference, not merely to its geometric support.
\end{remark}
\bibliographystyle{amsplain}

\begin{thebibliography}{10}
\bibitem{ACC} J. M. Ash, S. Catoiu, and M. Csörnyei, \textit{Generalized vs. ordinary differentiation}, Proc. Amer. Math. Soc., 145 (2017), 1553-1565. 
\bibitem{ACF1} J. M. Ash, S. Catoiu, and H. Fejzi\'{c}, \textit{Gaussian Riemann derivatives}, Israel J. Math. \textbf{255} (2023) (1), 177--199.

\bibitem{ACF} J. M. Ash, S. Catoiu, and H. Fejzi\'{c}, \textit{Two pointwise characterizations of the Peano derivative}, Results Math. \textbf{79} (2024) (7), Article No. 251,~17pp. 

\bibitem{CF1} S. Catoiu, and H. Fejzi\'{c}, \textit{NonEquivalence Between Peano and Riemann derivatives}, Israel J. Math. (to appear), 1--8.

\bibitem{CF2} S. Catoiu, and H. Fejzi\'{c}, \textit{A generalization of the GGR conjecture}, Proc. Amer. Math. Soc. 151 (2023), 5205-5221.

\bibitem{EW} M. J. Evans and C. E. Weil, \textit{Peano derivatives: A survey}, Real Anal. Exchange \textbf{7} (1981/82), no.~1, 5--23.

\bibitem{MZ} J. Marcinkiewicz and A. Zygmund, \textit{On the differentiability of functions and summability of trigonometric series,} Fund. Math. \textbf{26} (1936), 1--43.




\bibitem{P} G. Peano, \textit{Sulla formula di Taylor}, Atti Acad. Sci. Torino \textbf{27} (1891--92), 40--46.

\bibitem{R} B. Riemann, \textit{Ub\"er die Darstellbarkeit einer Funktion durch eine trigonometrische Reihe}, Ges. Werke, 2. Aufl., pp. 227--271. Leipzig, 1892.

\end{thebibliography}

\end{document}